\documentclass[onefignum,onetabnum]{siamart171218}

\input epsf
\usepackage{epsfig}
\usepackage{color}
\usepackage{amsfonts}

\numberwithin{equation}{section}
\usepackage{algorithmic,algorithm}

\newtheorem{example}{Example}[section]

\newsiamremark{remark}{Remark}
\newsiamremark{hypothesis}{Hypothesis}
\crefname{hypothesis}{Hypothesis}{Hypotheses}
\newsiamthm{claim}{Claim}

\newcommand{\curl}{\operatorname{curl}}
\renewcommand{\div}{\operatorname{div}}

\newcommand{\diag}{\operatorname{diag}}

\newcommand{\Range}{\operatorname{Range}}

\def\b1{{\mathbf 1}}
\def\bv{{\mathbf v}}
\def\bu{{\mathbf u}}
\def\bw{{\mathbf w}}

\def\br{{\mathbf r}}

\def\be{{\mathbf e}}

\def\bq{{\mathbf q}}
\def\bx{{\mathbf x}}

\def\bbf{{\mathbf f}}

\newcommand{\A}{{\mathcal A}}

\newcommand{\T}{{\mathcal T}}

\def\bxi{{\boldsymbol \xi }}

\ifpdf
\hypersetup{
  pdftitle={Modifying AMG Coarse Spaces to Exhibit Approximation in Energy},
  pdfauthor={X.~Hu and P.~S. Vassilevski}
}
\fi

\begin{document}

\title{Modifying AMG coarse spaces with weak approximation property to exhibit approximation in energy norm\thanks{This work was performed under the auspices of the U.S. Department
of Energy by Lawrence Livermore National Laboratory under Contract
DE-AC52-07NA27344. The work of the second author was partially supported by NSF under grant DMS-1619640. The work of Hu was partially supported by NSF grant DMS-1620063.}}

\author{Xiaozhe Hu \thanks{Department of Mathematics, Tufts University, Medford, MA 02155. 
  (\email{xiaozhe.hu@tufts.edu})}
\and Panayot S. Vassilevski\thanks{Center for Applied Scientific Computing,
             Lawrence Livermore National Laboratory,
             P.O. Box 808, L-561,
             Livermore, CA 94551, U.S.A. and 
             Department of Mathematics, Portland State University, Portland, OR 97201
  (\email{panayot@llnl.gov}, \email{panayot@pdx.edu}).}
  }

\maketitle

\begin{abstract}
Algebraic multigrid (AMG) coarse spaces are commonly constructed so that they exhibit
the so-called weak approximation (WAP) property which is necessary and sufficient condition for uniform two-grid convergence. This paper studies a modification of such coarse spaces so that the modified ones provide approximation in energy norm.  Our modification is based on the projection in energy norm onto an orthogonal complement of original coarse space. This generally leads to dense modified coarse space matrices which is hence computationally infeasible. To remedy this, based on the fact that the projection involves inverse of a well-conditioned matrix, we use polynomials to approximate the projection and, therefore, obtain a practical, sparse modified coarse matrix and prove that the modified coarse space maintains computationally feasible approximation in energy norm. We present some numerical results for both,  PDE discretization matrices as well as graph Laplacian ones, which are in accordance with our theoretical results.
\end{abstract}



\section{Introduction}\label{section: introduction}
Algebraic multigrid is one of the most successful methods for solving large-scale sparse systems of linear equations $A \bu = \bbf$ with symmetric positive definite (SPD) matrix $A$, especially for the case when $A$ comes from finite element discretization of second order elliptic equations.  AMG has also been extended to matrices arising from much broader classes of discretized PDEs (e.g., \cite{RS-AMG-edge-elements}, the AMS and ADS solvers in \cite{AMS}, \cite{ADS}) and even for non-PDE matrices (using adaptive AMG, see e.g., \cite{AdaptiveSA, DaV13}), including ones coming from network simulations (e.g. graph Laplacian, \cite{XXX}).  For an overview of some AMG methods, we refer to~\cite{MLBFP} and more recently to~\cite{Xu-Zikatanov: Acta Numerica}.  

Another important aspect of AMG, which is the main focus of this work, is that it provides a hierarchy of coarse spaces, which are natural candidates for dimension reduction, sometimes referred to as numerical {\em upscaling}.  
There are quite a few literature on multigrid-based upscaling techniques, e.g., \cite{GGR97,MM06,MDH98}, and domain-decomposition-based upscaling approaches, e.g., \cite{LSG09,SDHNPS14}.  However, one difficulty, which needs to be overcome with such an approach, is that the traditional AMG coarse spaces can not guarantee the required approximation accuracy.  More precisely, by the construction, traditional AMG coarse spaces only guarantee to possess a so-called {\em weak approximation property} (WAP), i.e., for any vector $\bu \in \mathbb{R}^n$, there exists a vector $\bu_c$ belonging to the coarse space, such that $\| \bu - P\bu_c \|_D \leq \eta_w \| \bu \|_A$, where $\| \bu \|_A := \sqrt{\bu^T A \bu}$ is the so-called energy norm and $\| \bu \|_D := \sqrt{\bu^T D \bu}$ is the (weighted) $\ell_2$-norm induced by a proper chosen SPD matrix $D$. The WAP is known to be necessary and sufficient for the uniform convergence of the two-level AMG methods (cf., e.g., \cite{MLBFP}). However, to use the same coarse space for dimension reduction, we need that the Galerkin projection (projection with respect to energy norm $\| \cdot \|_A$) onto the coarse space exhibit some approximation property.  A sufficient condition is that the coarse spaces satisfy the so-called {\em strong approximation property} (SAP), i.e., the coarse-level solution should approximate the original (fine-level) solution with some guaranteed accuracy in energy norm.  Mathematically, the SAP means that, for any vector $\bu \in \mathbb{R}^n$, there exists a vector $\bu_c$ belonging to the coarse space, such that $ \| A \| \| \bu - P \bu_c \|^2_A \leq \eta_s \| A \bu \|^2 $. Although the AMG coarse spaces do have approximation properties (by construction, in a weighted $\ell_2$-norm), the coarse-level solution (i.e., the computationally feasible Galerkin projection) does not generally possess that, neither in (weighted) $\ell_2$-norm nor in energy norm $\| \cdot \|_A$. To the best of our knowledge, none of the existing multigrid- and domain decomposition-based upscaling techniques have the desired SAP property with provable satisfactory bound on the resulting constant $\eta_s$. 

In this paper, we address the issue that the usual AMG coarse spaces do not satisfy the SAP with provable satisfactory bound on the resulting constant $\eta_s$ and develop an approach by extending a construction originated in \cite{MPe14} to our more general AMG upscaling setting. Our main contribution, which distinguish our result from all the existing results, is that our modified coarse space satisfies the SAP with provable satisfactory bound on the resulting constant $\eta_s$, which provides computable approximation to the fine-level solution in both the energy norm and (weighted) $\ell_2$-norm. The proposed method simply modifies the AMG coarse space $\Range(P)$ ($P$ is the prolongation matrix which satisfies the WAP by construction/assumption) to $\Range((I - \pi_f)P)$ where $\pi_f$ is a projection onto the $A$-orthogonal complement of $\Range(P)$ (i.e., orthogonal complement of $\Range(P)$ with respect to the $A$ inner product $\bu^TA \bv$).  We show that such modified coarse space provides a two-level $A$-orthogonal decomposition of the original fine-level solution $\bu$ and, thereby, energy error estimate of the coarse solution. Moreover, the SAP of the coarse space can be derived based on such decomposition as well. Details of the construction of $\pi_f$ will be presented in Section~\ref{section: modified coarse space}. Because the definition of $\pi_f$ involves the inverse of a matrix (see Section~\ref{sec:A-orth-complement} for details),  such modification typically leads to dense coarse matrices which is mostly of theoretical interest. In order to design a more practical approach, we take advantage of the fact that $A$ is well-conditioned on the $A$-orthogonal complement of $\Range(P)$ (which we prove holds for $P$ satisfying the WAP) and, therefore, modifying the coarse spaces based on polynomial approximations to control the sparsity of the respective coarse matrices is feasible. That is, we are able to modify the coarse space so that both, the SAP (hence the error estimate in the energy norm) and the sparsity of the coarse matrix, are satisfied. The energy error estimate improves  when the polynomial degree increases (with the expense of increased matrix density).  We present numerical results illustrating the effectiveness of the proposed method.  We would like to point out that other computationally feasible AMG-type upscaling approaches were presented in \cite{BLV} and \cite{Kalchev et al. 2016} for problems that can be formulated in a mixed (saddle-point) form. 

The remainder of the paper is structured as follows. In Section~\ref{section: WAP}, we introduce the WAP and formulate some properties of the matrices arising from the unsmoothed aggregation AMG. It provides the motivation for the construction of the improved coarse spaces which is presented in Section~\ref{section: modified coarse space}.  The error analysis in the computationally infeasible case with exact projections is presented in that section as well.  The computationally feasible case with approximate projections, giving rise to the improved coarse space satisfying the SAP and with guaranteed approximation properties is presented in Section~\ref{section: modified coarse space with approximate projection}.   The case of elliptic problems with high contrast coefficients is briefly discussed in Section~\ref{section: some remarks}.  The numerical illustration of the presented methods for both, PDE-type matrices and graph Laplacian ones, can be found in Section~\ref{section: numerical results}.  Finally, some conclusions are drawn in the last Section~\ref{section: conclusions}.

\section{Weak approximation property in AMG}\label{section: WAP}
In this section, we recall the two-grid method and the weak approximation property that is widely used to prove the convergence of two-grid methods. We point out that the prolongation matrices $P$ constructed in various AMG methods usually satisfy the WAP (a notion intoduced already in the original AMG paper, \cite{AMG original}.

We consider a SPD matrix $A \in \mathbb{R}^{n \times n}$ and let $D$ be another SPD matrix such that,
\begin{equation}\label{scaled D}
\bv^TA \bv \le \bv^T D \bv.
\end{equation}
A typical choice of $D$ is the diagonal of $A$ with proper scaling, i.e. $D = \omega^{-1} \, \text{diag}(A)$, $\omega \in \mathbb{R}$ or the so-called ``{\em $\ell_1$-smoother}'' (cf. e.g., \cite{brezina_vassilevski:2011}). We denote the norms induced by $A$ and $D$ by $\| \cdot \|_A$ and $\| \cdot \|_D$, respectively.

\subsection{The two-grid method}
First, we briefly recall the standard two-grid method.  Assume we have a smoother $M$ such as Jacobi, Gauss-Seidel, etc., a prolongation $P$, and the coarse-grid problem $A_c = P^T A P$.  
Based on these standard components, we define the standard (symmetrized) two-grid method in Algorithm~\ref{alg:TG}.
\begin{algorithm}
	\caption{Two-grid method} \label{alg:TG}
	For a current iterate $\bu$, we perform:
	\begin{algorithmic}[1]
		\STATE{Presmoothing:} $\bu \gets \bu + M^{-1} ( \bbf - A\bu)$
		\STATE{Restriction:} $\br_c \gets P^T(\bbf - A\bu)$
		\STATE{Coarse-grid correction:} $\be_c = A_c^{-1} \br_c$
		\STATE{Prolongation:} $\bu \gets \bu + P \be_c$
		\STATE{Postsmoothing:} $\bu \gets \bu + M^{-T} ( \bbf - A\bu)$
	\end{algorithmic}
\end{algorithm}

It is well-known that the two-grid method (Algorithm~\ref{alg:TG}) leads to the composite iteration matrix $E_{TG}$ based on which we define the two-grid operator $B_{TG}$ as follow,
\begin{equation*}
I - B_{TG}^{-1}A = E_{TG} = (I - M^{-T}A) (I - P A_c^{-1} P^TA) (I - M^{-1}A).
\end{equation*}

For the convergence rate of the two-grid method, we have the following two-grid estimates which can be found in Theorem 4.3, \cite{FVZ05}.  

\begin{theorem}\label{thm:TG-XZ}
	For $B_{TG}$ and the two-grid error propagation operator $E_{TG}$, we have the sharp estimates
	\begin{equation*}
	\bv^TA \bv \le \bv^TB_{TG}\bv\leq K_{TG} \bv^TA \bv \quad \text{or equivalently} \quad \| E_{TG} \|_{A} =  \rho_{TG} := 1 - \frac{1}{K_{TG}} ,
	\end{equation*} 
	where
	\begin{equation*}
	K_{TG} = \max_{\bv} \min_{\bv_c} \frac{\| \bv - P\bv_c \|^2_{\widetilde {M}}}{\| \bv \|^2_A},
	\end{equation*}
and ${\widetilde M}:=M^T(M^T+M-A)^{-1}M$ is the symmetrized smoother (starting with $M^T$).
\end{theorem}

\subsection{The Weak Approximation Property}
In AMG, we construct a prolongation $P \in \mathbb{R}^{n \times n_c}$ and the corresponding coarse space $\text{Range}(P)$ which exhibits the WAP. We note that the WAP is a necessary and sufficient condition for uniform two-level AMG convergence (e.g., \cite{MLBFP}) and can be stated as, for any vector $\bv \in {\mathbb R}^n$, there is a coarse vector $\bv_c \in {\mathbb R}^{n_c}$, such that
\begin{equation}\label{WAP}
\|\bv - P\bv_c\|_D \le \eta_w\;\|\bv\|_A,
\end{equation}
where $\eta_w$ is the so-called WAP constant. By requiring that the smoother is spectrally equivalent to $D$, which can be verified for standard smoothers such as Gauss-Seidel and Jacobi, we can estimate the two-grid constant $K_{TG}$ based on the WAP.  More precisely, we have $K_{TG} \leq c \eta^2_w$ where the constant $c$ here measures the spectral equivalence between $\widetilde{M}$ and $D$.  This implies that $\rho_{TG}\leq 1 - \frac{1}{c \eta_w^2}$, i.e., the corresponding two-grid method converges uniformly.

In order to have a computationally feasible approach (which will become clear later on), 
in this paper, we follow~\cite{brezina_vassilevski:2011,vassilevski:2011} and assume that $P$ is constructed based on aggregation-based approach (without smoothing).  Roughly speaking, we first form a set of aggregates $\{\A_i\}_{i=1}^{n_a}$, which is a nonoverlapping partitioning of the index set $\{1,\;2,\; \dots,\;n\}$, i.e., $\cup_{i=1}^{n_a} \A_i = \{ 1,\;2,\; \dots,\;n\}$ and $\A_i \cap \A_j = \emptyset$, if $i \neq j$.  Moreover, we denote the size of $\A_i$ by $n_{\A_i}$ which is defined by the cardinality of $\A_i$.  We then solve certain (generalized) eigenvalue problems locally to obtain the local basis $\{ \bq_{\A_i, j}^c \}_{j=1}^{n_i^c}$ for each aggregate $\A_i$.  The overall prolongation is defined as 
\begin{equation}\label{def:tent-P}
P = 
\begin{pmatrix}
P_{\A_1} & 0 & \cdots & 0 \\
0 & P_{\A_2} & \cdots &  \vdots \\
\vdots & \vdots & \ddots & \vdots\\
0 & \cdots & 0 & P_{\A_{n_a}} 
\end{pmatrix}
\ 
\text{with}
\
P_{\A_i} = (\bq_{\A_i, 1}^c,\cdots, \bq_{\A_i,n_i^c}^c),
\end{equation}
and, naturally, the coarse space is just $\text{Range}(P)$.  The WAP~\eqref{WAP} can be shown by the properties of the local eigenvalue problems.  We refer to~\cite{brezina_vassilevski:2011,vassilevski:2011} for the details. We note that such local spectral construction of $P$ \eqref{def:tent-P} dated back to \cite{AMGe spectral} and is also possible for graph Laplacian matrices (see, e.g., \cite{hu vassilevski and xu}). 

As already mentioned, a WAP of the above form is a necessary condition for uniform two-level AMG convergence, so
we assume \eqref{WAP} to hold for a block-diagonal $P$ and a diagonal $D$ (scaled as in \eqref{scaled D}).


The assumptions on $P$ and $D$ imply that the matrix $P^TDP$ is sparse, actually it is block diagonal with each diagonal block corresponding to an aggregate $\A_i$.  Hence, it is easily invertible and the projection $\pi_D = P (P^TDP)^{-1}P^TD$ is sparse, hence computationally feasible.  Taking $\bv_c = (P^TDP)^{-1}P^TD \bv$ in \eqref{WAP}, we arrive at the following estimate, which is another way to present the WAP of the coarse space using the projection $\pi_D$,
\begin{equation}\label{WAP with projection}
\|\bv - \pi_D \bv \|_D \le \eta_w\;\|\bv\|_A. 
\end{equation}

As already mentioned, the WAP plays an important role in the convergence analysis of  AMG methods.  For example, we can derive two-level convergence rate directly from the WAP.  However, in this paper, our goal is to take advantage of the WAP and modify the coarse space such that the modified one satisfies not only the WAP but also the so-called strong approximation property.  Coarse spaces that satisfy the SAP with provable satisfactory bound on the constant can provide a coarse-level solution which approximates the fine-level solution with guaranteed accuracy in energy norm,
and, therefore, are important both theoretically (e.g., in the V-cycle convergence analysis) and practically (e.g., for upscaling). 

\subsection{The $A$-Orthogonal Complement to $\Range(I-\pi_D)$} \label{sec:A-orth-complement}
Our modification of the coarse spaces (which will be presented in the next section) uses information from the orthogonal complement $\Range(I-\pi_D)$. Therefore, in this subsection, we introduce how to construct a sparse linearly independent basis of the space $\Range(I-\pi_D)$ and how to project a coarse vector onto it. 

The construction of the basis of the space $\Range(I-\pi_D)$ is, of course, not unique.  Here, we are looking for a sparse (locally supported) basis due to computational complexity considerations.  In the case of aggregation-based AMG, this can be done as follows.  On each aggregate $\A_i$, we select $n^f_i$ vectors, $\{\bq^f_{\A_i,j}\}^{n^f_i}_{j=1}$, which are orthonormal with respect to  $D_{\A_i} := \left .D \right |_{\A_i}$ and span the $D_{\A_i}$-orthogonal complement of $\Range(P_{\A_i})$.  Recall that $n_{\A_i}$ is the size of the aggregates $\A_i$ and $n_{i}^c$ be the number of columns of $P_{\A_i}$, we choose $n^f_i$ such that $n_{\A_i} = n^c_i + n^f_i$.  It is clear that the vectors $\bq^f_{\A_i,j}$ extended by zero outside $\A_i$ form a basis of $\Range(I-\pi_D)$. Introducing the matrix $P_\perp$ with the vectors $\bq^f_{\A_i,j}$ as its columns, then we have,
\begin{equation}\label{properties of P perp}
P^T_\perp D P =0 \text{ and }P^T_\perp D P_\perp = I.
\end{equation}

Exploiting the local basis of $\Range(I-\pi_D)$, we project any given vector $P\bv_c \in \text{Range}(P) $  onto the $A$-orthogonal complement of $\Range(I-\pi_D)$  by solving the following problem: 
find $\bv_f \in \Range(I-\pi_D)$, such that
\begin{equation}\label{A_f-projection}
(\bw_f)^T A \bv_f = (\bw_f)^T A P \bv_c,\quad \forall  \, \bw_f \in \Range(I-\pi_D).
\end{equation}
Since we have a sparse (computable) basis of $\Range(I-\pi_D)$ represented by $P_\perp$, we can rewrite~\eqref{A_f-projection} as the following linear system of equations,
\begin{equation}\label{matrix form of A_f-projection}
A_f {\overline \bv}_f = P^T_\perp A P \bv_c,
\end{equation}
where $A_f = P^T_\perp A P_\perp$ and $\bv_f = P_{\perp} \overline{\bv}_f$.  By solving~\eqref{matrix form of A_f-projection}, we compute the projection $\bv_f = \pi_f P \bv_c$. In fact, the matrix representation of $\pi_f$ is given by $\pi_f= P_\perp A^{-1}_f P^T_\perp A$. Note the inverse of $A_f$ is involved in the definition of $\pi_f$.

We next study the conditioning of $A_f$ with the goal to derive computationally feasible (sparse) approximations to its inverse within reasonable computational cost.  
We have the following main result.
\begin{theorem}\label{thm:Af-well-cond}
If the coarse space $\text{Range}(P)$ satisfies the WAP with constant $\eta_w$, 
then the condition number $\kappa(A_f)$ of $A_f$ satisfies $\kappa(A_f) \leq \eta_w^2$.
\end{theorem}
\begin{proof}
Choose $\bv = \bv_f := (I-\pi_D)\bv$ in \eqref{WAP with projection} and~\eqref{scaled D}, which leads to  the following spectral equivalence relations,
\begin{equation*}
\frac{1}{\eta_w^2}\; \bv^T_f D \bv_f \le \bv^T_f A  \bv_f \le \bv^T_f D \bv_f,  \quad \forall \, \bv_f \in \Range(I-\pi_D).
\end{equation*}
Equivalently, letting $ \bv_f = P_\perp {\overline \bv}_f $, using properties~\eqref{properties of P perp}, we have
\begin{equation}\label{spectral bounds}
\frac{1}{\eta_w^2}\; {\overline \bv}^T_f {\overline \bv}_f \le {\overline \bv}^T_f A_f  {\overline \bv}_f \le {\overline \bv}^T_f {\overline \bv}_f, \quad \forall \, {\overline \bv}_f,
\end{equation}
which implies that the condition number $\kappa(A_f)$ of $A_f$, satisfies $\kappa(A_f) \le \eta^2_w$, which is the desired result.
\end{proof}

\begin{remark}
When the WAP constant $\eta_w$ is bounded, especially independent of problem size, then Theorem~\ref{thm:Af-well-cond} implies that $A_f$ is well-conditioned.
\end{remark}

Since $A_f$ is well-conditioned, $A_f^{-1}$ can be accurately approximated by a matrix polynomial $q_\nu(A_f)$ of degree $\nu$.   Therefore, to approximate the solution of $A_f \bx_f = \bbf_f$, we can use the representation 
\begin{equation*}
\bx_f = A^{-1}_f  \bbf_f =  \left [ \left ( A^{-1}_f - q_\nu(A_f) \right ) \bbf_f  + q_\nu(A_f) \bbf_f\right ].
\end{equation*}
The first term on the right hand side above can be made as small as we want by choosing $q_\nu$ appropriately. 
 More precisely,  it can be made of order  $\epsilon \ll 1$ if we choose the polynomial degree $\nu  = \mathcal{O}(\log \, \epsilon^{-1})$ (cf., e.g., \cite{demko et al}, or \cite{MLBFP}, p.~413).  We give specific examples of polynomials $q_\nu(t)$ in Section~\ref{section: modified coarse space with approximate projection}.  By dropping the first term, we get the approximation 
\begin{equation} \label{approx:poly}
\bx_f \approx q_{\nu}(A_f) \bbf_f.
\end{equation}
 An important observation is that, if $\bbf_f$ is locally supported (sparse), the above approximation can be kept reasonably sparse.  In particular, consider~\eqref{matrix form of A_f-projection}, i.e., $\bbf_f = P_{\perp}^TA P\bv_c$, and let $\bv_c$ be one of the  unit coordinate vectors, then $P \bv_c$ is a column of $P$ and has local support represented by a corresponding aggregate $\A$.   Thus, such $\bbf_f$ is locally support on $\A$ and its immediate neighbors.  In this case, the approximation $q_\nu(A_f) \bbf_f $ is supported locally. More precisely, its suport depends on the sparsity of 
$A^{\nu}_f$, hence the  diameter  of the non-zero pattern of $q_\nu(A_f) \bbf_f $ can be estimated  to be of  order $\nu$ times the size of the  neighborhood of $\A$ and, therefore, can be kept under control  when $\nu$ is kept small.

The above approximation is the main motivation for our work.  Roughly speaking, such approximation allows us to modify the original  coarse space (with the WAP) so that the modified one satisfies the SAP while keeping the sparsity of the modified prolongation under control.  In the next two sections, we first introduce the SAP result in the case of exact $A^{-1}_f$  and then present the coarse space modification based on the computationally feasible polynomial approximation.

\section{The modified coarse space exhibiting the SAP}\label{section: modified coarse space}
In this section we define the modified coarse space. The construction presented here goes back to~\cite{MPe14}. In this paper, we adopt a matrix-vector presentation and motivate the applicability of the construction in ~\cite{MPe14} to our setting of aggregation-based  AMG exploiting the well-conditionedness  of $A_f$ proven in Theorem~\ref{thm:Af-well-cond}.  Thereby, we extend the analysis in~\cite{MPe14} to our more general (algebraic) setting by showing that the modified coarse spaces satisfy the SAP with provable satisfactory bound on the resulting constant $\eta_s$. 
In the following section, we extend these results to the case of approximate inverses. 

\subsection{Modification of the Coarse Space}
We first recall the projection $\pi_f = P_\perp A^{-1}_f P^T_\perp A$ which plays an important role in the construction of the modified coarse space. We also recall the original coarse space given by $\Range(P) = \Range(\pi_D)$.  The modified coarse space of our main interest is simply $\Range((I-\pi_f) \pi_D)$, or equivalently $\Range((I-\pi_f)P)$.  Naturally, the modified prolongation matrix 
takes the form $(I - \pi_f)P$.

Next, we show that we can obtain an A-orthogonal decomposition of any given vector $\bu$ based on the modified coarse space, which in turn implies the SAP of our main interest.   
To this end, we first present some properties of the two projections $\pi_D$ and $\pi_f$ summarized in the following lemma.
\begin{lemma}\label{lem:piD-pif}
The projections $\pi_D$ and $\pi_f$ satisfy $\pi_D \pi_f = 0$.  In addition, we have  that $(I-\pi_f)\pi_D$ is also a projection.
\end{lemma} 
\begin{proof}
$\pi_D \pi_f = 0$ can be directly verified by $\pi_D= P(P^TDP)^{-1}P^TD$ and $\pi_f = P_\perp A^{-1}_f P_{\perp}^T A$.  Together with properties~\eqref{properties of P perp}, we have 
$$
\pi_D \pi_f = P(P^TDP)^{-1} \underbrace{(P^TDP_\perp)}_{=0} A^{-1}_f P_{\perp}^T A  = 0.
$$ 
  On the other hand, using $\pi_D \pi_f = 0$ and also the fact that $\pi^2_D = \pi_D$, we have
\begin{align*}
\left ((I-\pi_f)\pi_D \right )^2 &= (\pi_D - \pi_f \pi_D) (\pi_D -\pi_f\pi_D)  \\
& \quad = \pi^2_D -\pi_f \pi^2_D  - (\pi_D \pi_f) \pi_D+ \pi_f (\pi_D \pi_f) \pi_D \\
& \quad = \pi_D -\pi_f \pi_D = (I-\pi_f)\pi_D,
\end{align*}
which implies that $(I-\pi_f)\pi_D$ is a projection.
\end{proof}

We are now ready to derive our main  two-level A-orthogonal decomposition.
\begin{theorem}\label{thm:two-level-A-orth-decomp}
For a given $\bu$, there exists a $\bv$, such that 
\begin{equation}\label{main decomposition}
\bu = (I-\pi_D)\bv  + (I-\pi_f) \pi_D\bu. 
\end{equation}
Also, the two components in the above decomposition are $A$-orthogonal.
\end{theorem}
\begin{proof}
We begin with the following  $A$-orthogonal decomposition
\begin{equation}\label{TL decomposition}
\bu = (I-\pi_D)\bv  + \bxi, \text{ where } \bxi \in \left (\Range(I-\pi_D)\right )^{\perp_A}.
\end{equation}
Given a $\bv_c$, from the definition of  $\bv_f = \pi_f P\bv_c$ in~\eqref{A_f-projection}, we have
\begin{equation*}
\bw_f^TA(I - \pi_f)P\bv_c = 0, \text{ for all $\bw_f \in \text{Range}(I - \pi_D)$}.
\end{equation*}
The latter identity implies that the $A$-orthogonal complement $\left (\Range  (I-\pi_D) \right) ^{\perp_A}$ of $\text{Range}(I - \pi_D)$ satisfies the relations
\begin{equation}\label{range representation}
\left (\Range  (I-\pi_D) \right) ^{\perp_A} = \Range \left( (I-\pi_f) P \right) = \Range \left( (I-\pi_f) \pi_D \right).
\end{equation}
This means that in \eqref{TL decomposition}, 
$\bxi = (I -\pi_f) \pi_D \bw$ for some $\bw$,  hence the A-orthogonal decomposition \eqref{TL decomposition}
can be rewritten as follows,
\begin{equation*}
\bu = (I-\pi_D)\bv  + (I-\pi_f) \pi_D\bw. 
\end{equation*}
Finally, using Lemma~\ref{lem:piD-pif} we have $\pi_D \bu = \pi_D (I-\pi_f) \pi_D\bw = \pi^2_D \bw = \pi_D \bw$, which shows ~\eqref{main decomposition}.
\end{proof}

The above A-orthogonal decomposition \eqref{main decomposition} basically provides an energy stable decomposition since
\begin{equation*}
\| \bu \|_A^2 = \| (I-\pi_D)\bv \|_A^2  + \| (I-\pi_f) \pi_D\bw \|_A^2.
\end{equation*}
This is essential in multilevel analysis.  In the following subsections, we prove the SAP for the modified coarse space
$\text{Range}((I-\pi_f)P)$ and also establish our first main error estimates, all based on this decomposition.


\subsection{The Strong Approximation Property}\label{section:SAP}
In this subsection, we show that the modified coarse space $\Range((I - \pi_f)P)$ satisfies the SAP with provable satisfactory bound on the constant $\eta_s$.  To this end, for given $\bbf$, we consider the solution $\bu$ of the following linear system,
\begin{equation}\label{original problem}
A \bu = \bbf.
\end{equation}
The corresponding modified coarse problem (also known as the upscaled problem) reads
\begin{equation}\label{upscaled problem}
P^T (I-\pi_f)^T  A (I-\pi_f) P \bu_c = P^T (I-\pi_f)^T \bbf.
\end{equation}
In order to show the SAP, we are interested in estimating the error $\be = \bu - (I-\pi_f) P \bu_c$ in the energy norm $\| \cdot \|_A$, more precisely, the estimate of $\|\bu - (I-\pi_f) P \bu_c\|_A$ in terms of $\|\bbf\| = \|A \bu\|$.  The main result is formulated in the following theorem.
\begin{theorem}\label{theorem: energy upscaling error estimate}
Assume the WAP \eqref{WAP} holds. Let $\be = \bu - (I-\pi_f) P \bu_c$ be the error between the fine-level solution $\bu$ of problem
\eqref{original problem} and the upscaled (coarse) solution ${\overline \bu}_c = (I-\pi_f) P \bu_c$ of \eqref{upscaled problem}. 
Then, the following energy error estimate holds:
\begin{equation}\label{upscaling error estimate in energy}
\|\be\|_A \le \eta_w \|D^{-\frac{1}{2}}A \bu\|.
\end{equation}
\end{theorem}
\begin{proof}
By the property of the Galerkin projection, we have that $\be =\bu - {\overline \bu}_c$ is $A$-orthogonal to~$\Range((I-\pi_f)P) = \Range((I-\pi_f)\pi_D) = \left (\Range(I-\pi_D)\right )^{\perp_A}$. Therefore, using the decomposition \eqref{main decomposition}, we have 
\begin{equation} \label{def:e}
\be = \bu - {\overline \bu}_c = (I- \pi_D) \bv, \, \text{ for some }\bv.
\end{equation}
Since ${\overline \bu}_c \in \Range((I-\pi_f)P)  =  \left (\Range(I-\pi_D)\right )^{\perp_A} $, we also have
\begin{equation*}
\|\be\|^2_A =  (\bu - {\overline \bu}_c)^TA (I-\pi_D)\bv = (A \bu)^T (I-\pi_D) \bv \le \| D^{-\frac{1}{2}}\bbf \| \|(I-\pi_D)\bv\|_D.
\end{equation*}
Using the weak approximation property \eqref{WAP with projection} for $\bv:= (I-\pi_D)\bv = \be$, we then obtain
\begin{equation*}
\|\be\|^2_A \le \|D^{-\frac{1}{2}} \bbf\| \eta_w \|(I-\pi_D)\bv\|_A = \eta_w \|D^{-\frac{1}{2}} \bbf\| \|\be\|_A,
\end{equation*}
which implies~\eqref{upscaling error estimate in energy}.
\end{proof}

From the energy error estimate~\eqref{upscaling error estimate in energy}, assuming that $D$ is well-conditioned, we have the following corollary also known as \emph{strong approximation property}.  

\begin{corollary}[Strong Approximation Property] \label{coro:SAP}
We have the following estimate
\begin{equation} \label{ine:SAP}
\|A\| \|\bu - {\overline \bu}_c\|^2_A \le  \|D\| \|\bu - {\overline \bu}_c\|^2_A \le \eta_s\; \|A \bu\|^2,
\end{equation}
where $\eta_s \le \|D\|\|D^{-1}\| \eta^2_w$, which is referred to as the SAP constant.  If $D$ is well-conditioned, then $\eta_s$ is bounded from above by a constant.
\end{corollary}

As we have shown, the modified coarse space $\Range (I - \pi_f) P$ satisfies the SAP with provable satisfactory bound on the constant $\eta_s$.  However, we want to point out that, the practical usage of this modified coarse space is limited since $\pi_f$ involves $A_f^{-1}$ which is dense in general.  In Section~\ref{section: modified coarse space with approximate projection}, we discuss how to use the polynomial approximation~\eqref{approx:poly} to modify the coarse space which can be used in practice with the SAP approximately satisfied.  

\subsection{A Weighted $\ell_2$--Error Estimate}\label{section: weighted ell-2 error estimate}
The estimate~\eqref{upscaling error estimate in energy} allows us to prove an $\ell_2$--error estimate, which is a direct application of the Aubin-Nitsche argument.  Let $\be = \bu - {\overline \bu}_c$ be the error and consider the following linear system,
\begin{equation*}
A \bw = D \be.
\end{equation*}
We have,
\begin{equation*}
\|\be\|^2_D = \be^T (D \be) = \be^T A \bw.
\end{equation*}
Since $\be$ is $A$-orthogonal to the modified coarse space $\Range(I-\pi_f)\pi_D$, we have, for ${\overline \bw}_c = (I-\pi_f)\pi_D \bw$ ,
\begin{equation*}
\|\be\|^2_D = \be^T A (\bw -{\overline \bw}_c)  \le \|\be\|_A \|\bw- {\overline \bw}_c\|_A.
\end{equation*}
Applying estimate~\eqref{upscaling error estimate in energy} to the error $\be_\bw := \bw-{\overline \bw}_c$ leads to
\begin{equation*}
\|\be\|^2_D  \le \|\be\|_A  \eta_w \|D^{-\frac{1}{2}} A \bw\| = \eta_w \|\be\|_A \|D^{\frac{1}{2}} \be\| = \eta_w \|\be\|_A \|\be\|_D.
\end{equation*}
This implies the desired weighted $\ell_2$-error estimate stated below.
\begin{theorem}\label{theorem: ell-2 upscaling error estimate}
Let $\be = \bu - (I-\pi_f)\pi_D \bu$ be the error between the solutions of the original fine-level problem \eqref{original problem} and the upscaled one \eqref{upscaled problem}. Then, the following weighted $\ell_2$--error estimate holds:
\begin{equation}\label{weighted ell-2 estimate}
\|\be\|_D \le \eta_w \|\be\|_A \le \eta^2_w \|D^{-\frac{1}{2}} A \bu\|.
\end{equation}
\end{theorem}

\section{Modified coarse space using approximate inverses}\label{section: modified coarse space with approximate projection}
In this section, we discuss how to use approximations to make the modified coarse spaces more practical.  The basic idea is based on the well-conditioning of $A_f$ as shown in Theorem~\ref{thm:Af-well-cond} which allows for uniform polynomial approximation~\eqref{approx:poly}.  We argue that such an approximation keeps the sparsity of the modified prolongation matrix under control while maintaining the approximation properties of the modified coarse space reasonably well.
These are properties that make the resulting modified coarse spaces appropriate for upscaling as well for efficient use in multigrid methods in practice.

\subsection{Modification via Polynomial Approximation}
We begin with one possible choice of polynomial approximation.  Recall that according to~\eqref{spectral bounds}, the spectrum of $A_f$ is contained in $[1/\eta_w^2, 1] \subset (0, 1]$.  Therefore, we want to chose a polynomial $p_\nu$ of degree $\nu \ge 1$, such that $p_\nu(0) =1$ and $tp^2_\nu(t)$ has a small maximum norm over the interval $t \in [0,1]$.  One choice is the polynomial used in the smoothed aggregation algebraic multigrid (SA-AMG). It is  defined via the Chebyshev polynomials of odd degree, $T_{2\nu+1}$, as follows:
\begin{equation}\label{SA polynomial}
p_\nu(t) = \frac{(-1)^\nu}{2\nu+1} \frac{T_{2\nu+1}(\sqrt{t})} {\sqrt{t}}.
\end{equation}
As is  well-known (e.g., shown in~\cite{brezina vanek vassilevski, MLBFP, hu vassilevski and xu}), this polynomial has the following property
\begin{equation}\label{def:SA-poly-small}
\max\limits_{t \in (0,1]} \sqrt{t}|p_\nu(t)| = \frac{1}{2\nu+1}.
\end{equation}

Since $p_\nu(0) =1$, $p_\nu(t) = 1 - t q_{\nu-1}(t)$, where $q_{\nu-1}$ is a polynomial of degree $\nu-1$.  We actually use $q_{\nu-1}(t)$ to approximate $A_f^{-1}$, namely
\begin{equation}\label{def:approx-Af-inverse}
A_f^{-1} \approx {\widetilde A}^{-1}_f \equiv q_{\nu-1}(A_f).  
\end{equation}
By rewriting~\eqref{def:approx-Af-inverse},  we get
\begin{equation*}
I - {\widetilde A}^{-1}_f A_f = I - q_{\nu-1}(A_f) A_f = p_{\nu}(A_f).
\end{equation*}
The $A_f$-norm of this matrix can be made arbitrarily small as $\nu \rightarrow \infty$ by the property~\eqref{def:SA-poly-small}.


Letting ${\widetilde \pi}_f := P_\perp {\widetilde A}^{-1}_f P^T_\perp A$, we define the modified prolongation matrix ${\widetilde P}$ as follows,
\begin{equation}\label{widetilde P}
{\widetilde P} :=  (I-{\widetilde \pi}_f) P.
\end{equation}
The corresponding modified coarse space is $\Range((I-{\widetilde \pi}_f) P) = \Range((I-{\widetilde \pi}_f) \pi_D)$.   Note that, if we choose $\nu$ properly (sufficiently large but fixed), the modified prolongation matrix ${\widetilde P}$ stays reasonably sparse and can be used in practice with nearly optimal computational cost.

We notice that, the formula  ${\widetilde P} = (I - P_\perp q_{\nu-1}(A_f) P^T_\perp A) P$, somewhat resembles the construction of prolongation matrices used in SA-AMG. More specifically, in SA-AMG, we have ${\widetilde P} := p_\nu(D^{-1}A) P$. This observation offers the possibility to construct new SA-AMG methods by choosing simple $P_\perp$ (for example, not necessarily spanning the entire complement of $\Range(P)$)
so that $A_f:= P^T_\perp A P_\perp$  and hence the resulting ${\widetilde P}$ and respective modified coarse level matrix ${\widetilde P}^TA{\widetilde P}$ be reasonably sparse. 

\begin{remark}\label{remark: AWHB method}
We may also note that~${\widetilde P} = (I - P_\perp q_{\nu-1}(A_f) P^T_\perp A) P$ resembles the so-called {\em approximate wavelet modified hierarchical basis (AWMHB) method} where
$P_\perp$ (corresponding to the HB) is modified by polynomially based approximate $L_2$-projections to exhibit better energy stability (cf. \cite{vassilevski review} or \cite{MLBFP}).
\end{remark}

With the approximate modified coarse space, the two-level decomposition can be rewritten in the following perturbation form
\begin{equation}\label{approx-decomp}
\bu = (I-\pi_D)\bv  + (I-{\widetilde  \pi}_f) \pi_D\bu + ({\widetilde \pi_f} - \pi_f) \pi_D \bu. 
\end{equation}
Obviously, we do not have $A$-orthogonality anymore.  However, as we show later, the first two terms of the decomposition~\eqref{approx-decomp} are approximately $A$-orthogonal whereas the last term can be made small, which leads to the desired error estimates.

\subsection{Approximate Orthogonality}
To show that the first two terms of the decomposition~\eqref{approx-decomp} are approximately $A$-orthogonal, we prove that the two spaces $\Range(P_{\perp})$ ($=\Range(I - \pi_D)$) and $\Range({\widetilde P})$ ($=\Range( (I - {\widetilde \pi}_f) \pi_D)$) are approximately $A$-orthogonal.  To this end, we first establish some properties of $\pi_D$ and ${\widetilde \pi}_f$ summarized in the following lemma.  
\begin{lemma}\label{lem:piD-tildepif}
We have $\pi_D {\widetilde \pi}_f = 0$ and that $(I- {\widetilde \pi}_f)\pi_D$ is a projection.
\end{lemma} 
\begin{proof}
The proof is the same as the proof of Lemma~\ref{lem:piD-pif}.
\end{proof}

\begin{remark}
Lemma~\ref{lem:piD-tildepif} actually holds for ${\widetilde \pi}_f$ obtained by approximating $A_f^{-1}$ with any  ${\widetilde A}^{-1}_f$ in the definition of $\pi_f$.  Therefore, this allows us to use, for example,  other polynomials, i.e.,  not only the SA polynomial~\eqref{SA polynomial}.
\end{remark}

Next, we estimate the cosine of the abstract angle between the two spaces. For any vectors $\bv_f$ and $\bv_c$ and use the property~\eqref{def:SA-poly-small} of the SA polynomial~\eqref{SA polynomial}, we have
\begin{equation} \label{ine:cos}
\begin{array}{rl}
\bv^T_f P^T_\perp A {\widetilde P}\bv_c & = \bv^T_f P^T_\perp A \left (I -P_\perp {\widetilde A}^{-1}_f P^T_\perp A \right ) P\bv_c\\
& = \left ((I- A P_\perp {\widetilde A}^{-1}_f P^T_\perp) A P_\perp \bv_f \right )^T P \bv_c\\
& = \left (A P_\perp (I - {\widetilde A}^{-1}_f P^T_\perp  A P_\perp) \bv_f \right )^T P \bv_c \\
& = \left (P_\perp (I- {\widetilde A}^{-1}_f A_f)\bv_f \right )^T A P\bv_c\\
& = \left (P_\perp p_\nu(A_f) \bv_f \right )^T A P \bv_c \\
& \le \sqrt{\bv^T_f A_f p^2_\nu(A_f) \bv_f } \sqrt{ \bv^T_c P^T A P\bv_c }\\
& \le \max_{t \in (0,1]} \sqrt{t} |p_\nu(t)|\; \|\bv_f\| \|P \bv_c\|_A\\
& \le \frac{1}{2\nu+1}\; \|\bv_f\| \|P \bv_c\|_A.
\end{array}
\end{equation}
Given $\bw$ and $\bv$, consider $P \bv_c = \pi_D \bw$ and $P_{\perp} \bv_f = (I-\pi_D)\bv$.  Then, from~\eqref{ine:cos} and use the facts that $\|\bv_f\| = \|(I-\pi_D)\bv\|_D$, $\|P \bv_c\|_A = \|\pi_D\bw\|_A$,  and ${\widetilde P} \bv_c = (I-{\widetilde \pi}_f )\pi_D\bw$, to obtain 
\begin{equation} \label{ine:cos-vw}
((I-\pi_D)\bv)^T A (I-{\widetilde \pi}_f) \pi_D\bw \le \frac{1}{2\nu+1}\|(I-\pi_D)\bv\|_D \|\pi_D \bw\|_A.
\end{equation}
From \eqref{scaled D}, the WAP~\eqref{WAP with projection}, we have
\begin{equation*}
\|(I-\pi_D)\bv\|_A \le \|(I-\pi_D)\bv\|_D \le \eta_w \|\bv\|_A \text{ and } 
\end{equation*}
hence by  Kato's Lemma (\cite{MLBFP}), 
\begin{equation}\label{Kato}
\|\pi_D\|_A = \|I -\pi_D\|_A \le \eta_w.
\end{equation}
The latter estimates together with~\eqref{ine:cos-vw} imply,
\begin{equation}\label{almost orthogonality_0}
((I-\pi_D)\bv)^T A (I-{\widetilde \pi}_f) \pi_D\bw \le \frac{\eta^2_w}{2\nu+1}\; \|\bv\|_A\|\bw\|_A.
\end{equation}
This gives us the desired approximate $A$-orthogonality result stated below.
\begin{theorem}\label{theorem: almost orthogonality}
Assume the SA polynomial~\eqref{SA polynomial} is used to define ${\widetilde \pi}_f$, then the approximate modified coarse space $\Range({\widetilde P})$ ($=\Range( (I - {\widetilde \pi}_f) \pi_D)$) and the hierarchical complement $\Range(P_\perp)$ ($=\Range(I - \pi_D)$) of the original coarse space $\Range(P)$ are almost $A$-orthogonal in the following sense,
\begin{equation}\label{almost orthogonality}
((I-\pi_D)\bv)^T A (I-{\widetilde \pi}_f) \pi_D\bw 
\le \frac{\eta^2_w}{2\nu+1}\; \|(I-\pi_D)\bv\|_A \|(I-{\widetilde \pi}_f) \pi_D\bw\|_A.
\end{equation}
\end{theorem}
\begin{proof}
Apply \eqref{almost orthogonality_0} for $\bv:=(I-\pi_D)\bv$ and $\bw: = (I-{\widetilde \pi}_f) \pi_D\bw$ and use the facts that both $\pi_D$ and $(I-{\widetilde \pi}_f) \pi_D$ are projections. 
\end{proof}


\subsection{Energy Error Estimate}
The second result we prove is an energy error estimate using the approximate modified coarse space $\Range({\widetilde P})$.  We start with the following lemma which shows that the third term in the perturbed decomposition~\eqref{approx-decomp} is small.
\begin{lemma}\label{lem:small-3rd-term}
Assume the SA polynomial~\eqref{SA polynomial} is used to define ${\widetilde \pi}_f$, then we have
\begin{equation}\label{ine:small-3rd-term}
\|({\widetilde \pi}_f - \pi_f) \pi_D\bu\|_A \leq  \frac{\eta^2_w}{2\nu+1}\; \|\bu\|_A
\end{equation}
\end{lemma}
\begin{proof}
Let $\pi_D \bu =P\bu_c$, and consider the deviation term
\begin{equation}\label{deviation identity}
\begin{array}{rl}
\|({\widetilde \pi}_f - \pi_f) \pi_D\bu\|_A & = \|P_\perp ({\widetilde A}^{-1}_f - A^{-1}_f) P^T_\perp A P \bu_c\|_A\\
& = \|A^{\frac{1}{2}}_f \left ((I-p_\nu(A_f))A^{-1}_f - A^{-1}_f \right ) P^T_\perp A P \bu_c\|\\
& = \| p_\nu(A_f) A^{-\frac{1}{2}}_f P^T_\perp A P \bu_c\|\\
&\le \|p_\nu(A_f)  A^{-\frac{1}{2}}_f  P^T_\perp  A^{\frac{1}{2}}\| \|P\bu_c\|_A\\
& = \|A^{\frac{1}{2}} P_\perp A^{-\frac{1}{2}}_f p_\nu(A_f) \| \|P\bu_c\|_A\\
& = \|p_\nu(A_f) \| \|P\bu_c\|_A.
\end{array}
\end{equation}
For the SA polynomial~\eqref{SA polynomial}, using the fact that $\lambda_{\min}(A_f) \ge \frac{1}{\eta^2_w}$ (see \eqref{spectral bounds}) and $\|\pi_D\|_A \le \eta_w$, \eqref{Kato},  we have
\begin{equation*}
\begin{array}{rl}
\|({\widetilde \pi}_f - \pi_f) \pi_D\bu\|_A
& \le \frac{1}{\sqrt{\lambda_{\min}(A_f)}}
 \max\limits_{t \in [0,1]} \sqrt{t}|p_\nu(t)| \; \|P \bu_c\|_A\\
& \le \frac{\eta_w}{2\nu+1}\; \|P\bu_c\|_A \\
& = \frac{\eta_w}{2\nu+1}\; \|\pi_D \bu \|_A \\
& \le \frac{\eta^2_w}{2\nu+1}\; \|\bu\|_A,
\end{array}
\end{equation*}
which completes the proof.
\end{proof}

Consider the modified coarse problem based on the approximate inverse ${\widetilde A}^{-1}_f$ in ${\widetilde P}$, as follows 
\begin{equation*}
{\widetilde P}^T A {\widetilde P}  {\widetilde \bu}_c  = {\widetilde P}^T\bbf.
\end{equation*}
Let ${\widetilde \bu} = {\widetilde P} {\widetilde \bu}_c \in \Range({\widetilde P})$  be the respective coarse (upscaled) solution.
We have the following energy error estimate which is an extension of energy error estimate~\eqref{upscaling error estimate in energy}.
\begin{theorem}\label{theorem: energy error estimates}
If $p_\nu$ is the SA polynomial~\eqref{SA polynomial}, then the following energy error estimate holds
\begin{equation}\label{linear error decay}
\|\bu - {\widetilde P}  {\widetilde \bu}_c \|_A \le 
\|\be\|_A + \|({\widetilde \pi}_f - \pi_f) \pi_D\bu\|_A
\le \eta_w \|D^{-\frac{1}{2}} A \bu\| + \frac{\eta^2_w}{2\nu+1}\;\|\bu\|_A,
\end{equation}
with the perturbation term (last term on the right hand side) exhibiting linear decay in $\nu$.

\end{theorem}
\begin{proof}
Since the coarse solution is the best approximation to the solution $\bu$ of the original linear system~\eqref{original problem} from the modified coarse space in the $A$-norm, we have
\begin{equation*}
\|\bu - {\widetilde P}  {\widetilde \bu}_c \|_A = \min\limits_{\bv_c}\; \|\bu - {\widetilde P}\bv_c\|_A.
\end{equation*}
Note that $(I - {\widetilde \pi}_f) \pi_D \bu \in \Range({\widetilde P})$, then we have 
\begin{equation*}
\|\bu - {\widetilde P}  {\widetilde \bu}_c \|_A = \min\limits_{\bv_c}\; \|\bu - {\widetilde P}\bv_c\|_A \leq \|\bu- (I- {\widetilde \pi}_f) \pi_D\bu\|_A.
\end{equation*}
Hence, according to the decomposition~\eqref{approx-decomp} and~$\be = (I -\pi_D) \bv$ in~\eqref{def:e}, we have
\begin{equation*}
 \|\bu - {\widetilde P}  {\widetilde \bu}_c \|_A \le \|\bu- (I- {\widetilde \pi}_f) \pi_D\bu\|_A = \|\be + ({\widetilde \pi}_f - \pi_f) \pi_D\bu\|_A\le 
\|\be\|_A + \|({\widetilde \pi}_f - \pi_f) \pi_D\bu\|_A.
\end{equation*}
Then apply Theorem~\ref{theorem: energy upscaling error estimate} to the first term and Lemma~\ref{lem:small-3rd-term} to the second term, to arrive at~\eqref{linear error decay}.
\end{proof}

Further, assume that $D$ is well-conditioned, we have the following approximate the SAP, which is a perturbation of  Corollary~\eqref{coro:SAP}.
\begin{corollary}
If $p_\nu$ is the SA polynomial~\eqref{SA polynomial}, then
\begin{equation*}
\| A \|^{\frac{1}{2}} \| \bu - {\widetilde P} {\widetilde \bu}_c \|_A \leq \| D \|^{\frac{1}{2}} \| \bu - {\widetilde P} {\widetilde \bu}_c \|_A  \leq \eta_s^{\frac{1}{2}} \| A \bu \| + \frac{\eta_w \| D \|^{\frac{1}{2}} }{2 \nu + 1} \| \bu \|_A.
\end{equation*}
where~$\eta_s \le \|D\|\|D^{-1}\| \eta^2_w$.  If $D$ is well-conditioned, then $\eta_s$ is bounded above by a constant.
\end{corollary}

\subsection{Other Approximations}
The SA polynomial~\eqref{SA polynomial} is just one possible choice for approximating $A_f^{-1}$.  There are other possible choices as well.  In this subsection, we briefly discuss other possibilities.

If we have the WAP constant $\eta_w$ explicitly available, that is, we have explicit eigenvalue bounds, 
 $\lambda(A_f) \in [\alpha,\;\beta] \subset [\frac{1}{\eta^2_w},\; 1]$, we can use the (best) Chebyshev polynomial
\begin{equation}\label{Chebyshev polynomial}
p_\nu(t) = \frac{T_\nu\left ( \frac{\beta+\alpha - 2t}{\beta-\alpha}\right )}{T_\nu\left ( \frac{\beta+\alpha}{\beta-\alpha}\right )}.
\end{equation}
Then due to the optimality property of Chebyshev polynomial,
\begin{equation*}
\|p_\nu(A_f)\| \le \frac{2q^\nu}{1+q^{2\nu}}, \quad q = \frac{\eta_w -1}{\eta_w+1},
\end{equation*}
together with the identity \eqref{deviation identity},
we end up with the following error estimate. 
\begin{theorem}\label{thm:Cheb-error}
If $p_\nu$ is the Chebyshev polynomial~\eqref{Chebyshev polynomial} used to define the approximate modified coarse space 
$\text{Range}({\widetilde P})$, then the following energy error estimate holds
\begin{equation}\label{geometric error decay}
\|\bu - {\widetilde P}  {\widetilde \bu}_c \|_A \le  \eta_w \|D^{-\frac{1}{2}} A \bu\| + \frac{2q^\nu \eta_w}{1+q^{2\nu}}\;\|\bu\|_A,
\end{equation}
where now the perturbation term exhibiting geometric decay in $\nu$.

\end{theorem}


It is clear that error estimate \eqref{geometric error decay} is much better than \eqref{linear error decay}. 
We note that in the spectral AMGe method in the form presented in \cite{brezina_vassilevski:2011}, explicit bounds of $\eta_w$ are available.  Therefore, the Chebyshev polynomial~\eqref{Chebyshev polynomial} can be used to modify the coarse space in the spectral AMGe setting.

Using either the SA polynomial~\eqref{SA polynomial} or the Chebyshev polynomial~\eqref{Chebyshev polynomial} basically provides an approximate solution to the linear system~\eqref{matrix form of A_f-projection}.  Therefore, another way to solve~\eqref{matrix form of A_f-projection} is via nonlinear iterative methods such as the conjugate gradient (CG) method.  Using CG implicitly constructs a polynomial $p_{\nu}(t)$ which  defines ${\widetilde \pi}_f$. The convergence analysis of CG can be used to estimate $\|({\widetilde \pi}_f - \pi_f) \pi_D\bu\|_A$.  Denote the $\nu$-th iteration of CG for solving $A_f {\overline \bu}_c = P_{\perp}^T A P \bu_c$ by ${\overline \bu}_c^{\nu}$ with zero initial guess, then similarly to~\eqref{deviation identity}, we have
\begin{align*}
\|({\widetilde \pi}_f - \pi_f) \pi_D\bu\|_A 
& = \|P_\perp ({\widetilde A}^{-1}_f - A^{-1}_f) P^T_\perp A P \bu_c\|_A = \| ({\widetilde A}^{-1}_f - A^{-1}_f) P^T_\perp A P \bu_c \|_{A_f}\\
& = \| {\overline \bu}_c^{\nu} - {\overline \bu}_c \|_{A_f} \leq 2 q^{\nu} \| {\overline \bu}_c \|_{A_f}  =  2 q^{\nu} \| A_f^{-1} P_{\perp}^T A P \bu_c \|_{A_f}  \\
&\le 2 q^{\nu} \|A^{-\frac{1}{2}}_f  P^T_\perp  A^{\frac{1}{2}}\| \|P\bu_c\|_A = 2 q^{\nu} \|A^{\frac{1}{2}} P_\perp A^{-\frac{1}{2}}_f  \| \|P\bu_c\|_A\\
& = 2 q^{\nu} \|P\bu_c\|_A \leq 2  q^{\nu} \eta_w \| \bu \|_A.
\end{align*}
Therefore, we have the following result. 
\begin{theorem}\label{thm:CG-error}
If $p_\nu$ is the polynomial generated by CG, then the following energy error estimate holds
\begin{equation}\label{ine:CG-error}
\|\bu - {\widetilde P}  {\widetilde \bu}_c \|_A \le  \eta_w \|D^{-\frac{1}{2}} A \bu\| +  2q^\nu \eta_w\;\|\bu\|_A,
\end{equation}
with the perturbation term exhibiting geometric decay in $\nu$.

\end{theorem}

\begin{remark} \label{rem:CG-best}
The error estimates~\eqref{geometric error decay} and~\eqref{ine:CG-error} both have perturbation terms that decay geometrically with the same rate $q$, therefore, we can conclude that modifying the coarse space based on CG polynomial gives better estimates than the SA polynomial.  Note that the CG approximation also, as in the SA case,  does not need estimates for the spectrum of $A_f$, whereas these are needed in the Chebyshev polynomial case.
\end{remark}

\subsection{Example: Linear Finite Elements for Laplace Equation}
As a simple example, we consider the Laplace equation, $-\Delta u = f$, discretized using piecewise linear finite elements.  
In this case, we have $\eta_w \simeq \frac{H}{h}$ (cf., \cite{brezina_vassilevski:2011}) where $H$ stands for the diameter of the aggregates.  This fact, combined with a simple argument relating the right hand side of the discrete problem, $\bbf$,  and the $L_2$-norm $\|f\|_0$ of the right hand side function $f$ (as shown in \cite{vassilevski:2011}),  
we conclude that the first term $\eta_ w \|D^{-\frac{1}{2}}A \bu\| \simeq H\| f \|_0$. If we want to balance the second term with the first one, we need to choose 
$2q^\nu \simeq H$ (assume Chebyshev polynomial or CG used). This implies that 
\begin{equation*}
\nu\log \left (1 + \frac{2}{\eta_w-1} \right ) \simeq \log \frac{1}{H},
\end{equation*}
and since $\log \left (1 + \frac{2}{\eta_w-1} \right ) \simeq \frac{2}{\eta_w-1} \simeq \frac{h}{H}$, we have the following estimate for the polynomial degree (or the number of iterations used for CG)
\begin{equation*}
\nu \simeq \frac{H}{h} \log \frac{1}{H}.
\end{equation*}
This ensures the  error estimate,
\begin{equation*}
\|\bu - {\widetilde P}  {\widetilde \bu}_c \|_A \le  C \left (H\| f\|_0 + H \|\bu\|_A \right ).
\end{equation*}
Similar argument can also be applied to the SA polynomial case in order to get an estimate of the polynomial degree.

\section{Remarks for Elliptic Problems with High Contrast Coefficients}\label{section: some remarks}
We consider the case with exact projection $\pi_f$ for simplicity in this section.  In section~\ref{section: modified coarse space}, we showed that the second component of the two-level $A$-orthogonal decomposition 
\begin{equation*}
\bu = (I-\pi_D)\bv + (I-\pi_f)\pi_D \bu,
\end{equation*}
is actually the solution ${\overline \bu}_c$ of the modified coarse problem~\eqref{upscaled problem}. 
It is worth noticing the the first component above, $(I-\pi_D) \bv$, is the $A$-orthogonal projection of $\bu$ onto the space $\Range(I-\pi_D)$. We already discussed the fact that the matrix of this problem is sparse and well-conditioned (after symmetric diagonal scaling of $A$). Thus it is computationally feasible to explicitly compute this component as well. 
Of course, this is not surprising since a two-grid AMG with the standard coarse space $\Range(P)$ and using $D$ as a smoother is uniformly convergent, hence $\bu$ can be approximated well by a few V-cycles. Note that such an AMG uses only sparse matrix-operations with much sparser matrices than the one of the upscaled problem~\eqref{upscaled problem} and~$A_f$.  Therefore, introducing the modified coarse space  $\Range((I-\pi_f)\pi_D)$ and the resulting error estimate  \eqref{upscaling error estimate in energy} (and its corollaries) are mostly of theoretical value. In the case of approximate projections,  if we cannot control the sparsity of the coarse matrices so that the resulting method requires much less memory and computational cost than the original matrix $A$, then the upscaled problem is  mostly of theoretical value only. 
With our numerical tests we demonstrate that in the PDE case, careful choice of the polynomial degree can lead to some savings in practice for the upscaled problems. The situation for graph Laplacian matrices is more challenging for graphs with irregular degree
distribution. 

One possible practical application of the presented method is the diffusion equation,
\begin{equation} \label{eqn:jump}
\begin{cases}
&-\div (\kappa \nabla u) = f, \quad \text{in} \, \Omega, \\ 
&u = 0, \quad \text{on}\, \partial \Omega,
\end{cases}
\end{equation}
where~$\Omega \subset {\mathbb R}^d$, $d=2$ or $d=3$, is a polygonal/polyhedral domain.  Using $H^1$-conforming finite element space on a quasiuniform mesh $\T_h$, we end up with linear system of the form
\begin{equation*}
A \bu = \bbf.
\end{equation*}
By construction, we have $\|\bbf\| \simeq h^{\frac{d}{2}} \|  f\|_0$. Let $D$ be the diagonal of $A$. We have 
$D \simeq \diag(h^{d-2} \kappa_i)$, where $\kappa_i h^d$ are the diagonal entries of the weighted mass matrix corresponding to the $\kappa$-weighted $L_2$-bilinear form. Hence, we have the estimate 
\begin{equation*}
\|D^{-\frac{1}{2}} \bbf\| \le \eta_b\; h\| f\|_{0,\; \kappa^{-1}}
\end{equation*}
for a uniform constant $\eta_b$, which leads to the error estimate
\begin{equation*}
\|u_h-u_H\|_{1,\;k} \le \eta_w \eta_b\; h\| f\|_{0,\;\kappa^{-1}}.
\end{equation*}
Here $u_h$ is the finite element solution of the fine-grid problem and $u_H$ is the finite element solution corresponding to the upscaled solution~${\overline \bu}_c$ of~\eqref{upscaled problem}.  Note that, this error estimate is independent of the coefficient $\kappa$ with the expense of the weighted norms involved.  For $\kappa \simeq 1$, using the fact that $\eta_w \simeq H/h$, the last error estimate reads 
$\|u_h-u_H\|_1 \le CH \|f\|_0$ which is an analog to the one in \cite{MPe14}.

\section{Numerical Experiments}\label{section: numerical results}
In this section, we present numerical results illustrating the theory demonstrating the approximation properties of the modified coarse spaces. In all experiments, we use the AMGe method in the form proposed in~\cite{brezina_vassilevski:2011,vassilevski:2011} to construct the original coarse space~$\Range(P)$ so that it satisfies the WAP.  More presicely, we use a  greedy type algorithm to construct a set of aggregates and solve a generalized eigenvalue problem (see (13) in~\cite{brezina_vassilevski:2011}) to construct the tentative prolongation as shown in~\eqref{def:tent-P}. To assess the quality of the proposed approach in practice, we only consider two-grid method and the modified coarse space based on the polynomial approximations as discussed in Section~\ref{section: modified coarse space with approximate projection}. In fact, we use the CG polynomial in all our experiments as it gives the best error estimates (see Remark~\ref{rem:CG-best}).  
The tests are run in Matlab using an AMG package developed by the authors. 

\begin{example} \label{exp:simple-poisson}
Consider the diffusion problem~\eqref{eqn:jump} posed on $\Omega = [0,1] \times [0,1]$ with 
\begin{equation*}
\kappa  =
\begin{cases}
\epsilon, \quad  \text{\rm in } [0.25, 0.5] \times [0.25, 0.5] \cap [0.5, 0.75] \times [0.5, 0.75] \\
1, \quad \text{\rm otherwise}. 
\end{cases}
\end{equation*}
\end{example}

Our first example is diffusion problem~\eqref{eqn:jump} with discontinuous coefficient.  As discussed in Section~\ref{section: some remarks}, the modified coarse space provides error estimates that are independent of the jumps.  The results shown in Figure~\ref{fig:jump-independent} support this theoretical results. Here, the fine level problems are all of size $4,225 \times 4,225$ on a uniform triangular mesh with $h=\frac{1}{64}$ and the coarse level matrices are all of size $302 \times 302$.  We change the contrast of the diffusion coefficient, i.e., $\epsilon$, and report how the SAP constant $\eta_s$ changes with respect to the degree~$\nu$ of the polynomial (since we use CG, the degree is equivalent to the number of iterations). For comparison, we also report the SAP constants when we modify the coarse space exactly by directly inverting $A_f$. As clearly seen, the SAP constant
stays almost the same for different choices of~$\epsilon$ for a fixed~$\nu$, and is indeed practically independent of the contrast~$\epsilon$.  This is consistent with the theory and shows that the modified coarse spaces provide approximations in the energy norm that are robust with respect to the jumps.  From Figure~\ref{fig:jump-independent},  we also observe that the SAP constant decreases to the SAP constant that corresponds to the modified coarse space with exact inverse, when $\nu$ increases with a rate that is almost the same for different~$\epsilon$.  This is also consistent with the theoretical results presented in Section~\ref{section: modified coarse space with approximate projection}; namely, that the decay rate should depend on $\eta_w$ which, in fact, in the present case depends on $\frac{H}{h}$. 
 Our next numerical experiment further verifies this property; see the results shown in Figure~\ref{fig:jump-ratio}.  
Since $h$ is $N^{-1/2}$ and $H$ is roughly $N_c^{-1/2}$, we present the results in terms of the ratio $\frac{N}{N_c}$, which is roughly $\left(\frac{H}{h}\right)^2$.  
More specifically, from Figure~\ref{fig:jump-ratio}, we see that the SAP constant decreases when $\nu$ increases and the bigger the ratio $\frac{N}{N_c}$ is, the slower the decay rate is.  But, the SAP constant converges to the SAP constant corresponding to the modified coarse space with exact $A_f$  inverse,  as expected.

\begin{figure}[htbp]
\begin{center}
\includegraphics[width=0.9\textwidth]{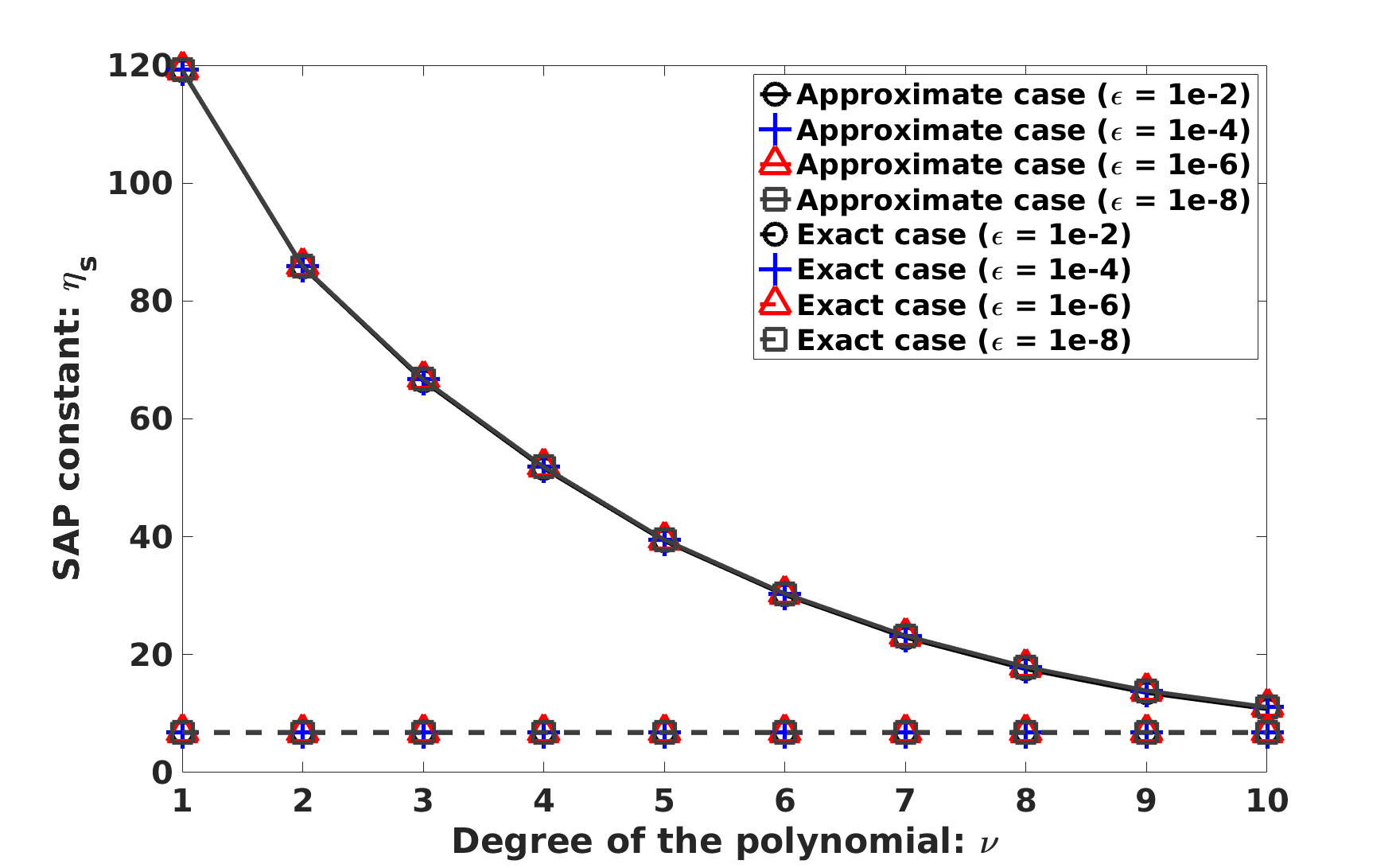}
\caption{Example~\ref{exp:simple-poisson}: the SAP constants for different~$\epsilon$ ($h=1/64$, $N = 4,225$ and $N_c = 302$)}
\label{fig:jump-independent}
\end{center}
\end{figure}

\begin{figure}[htbp]
\begin{center}
\includegraphics[width=0.9\textwidth]{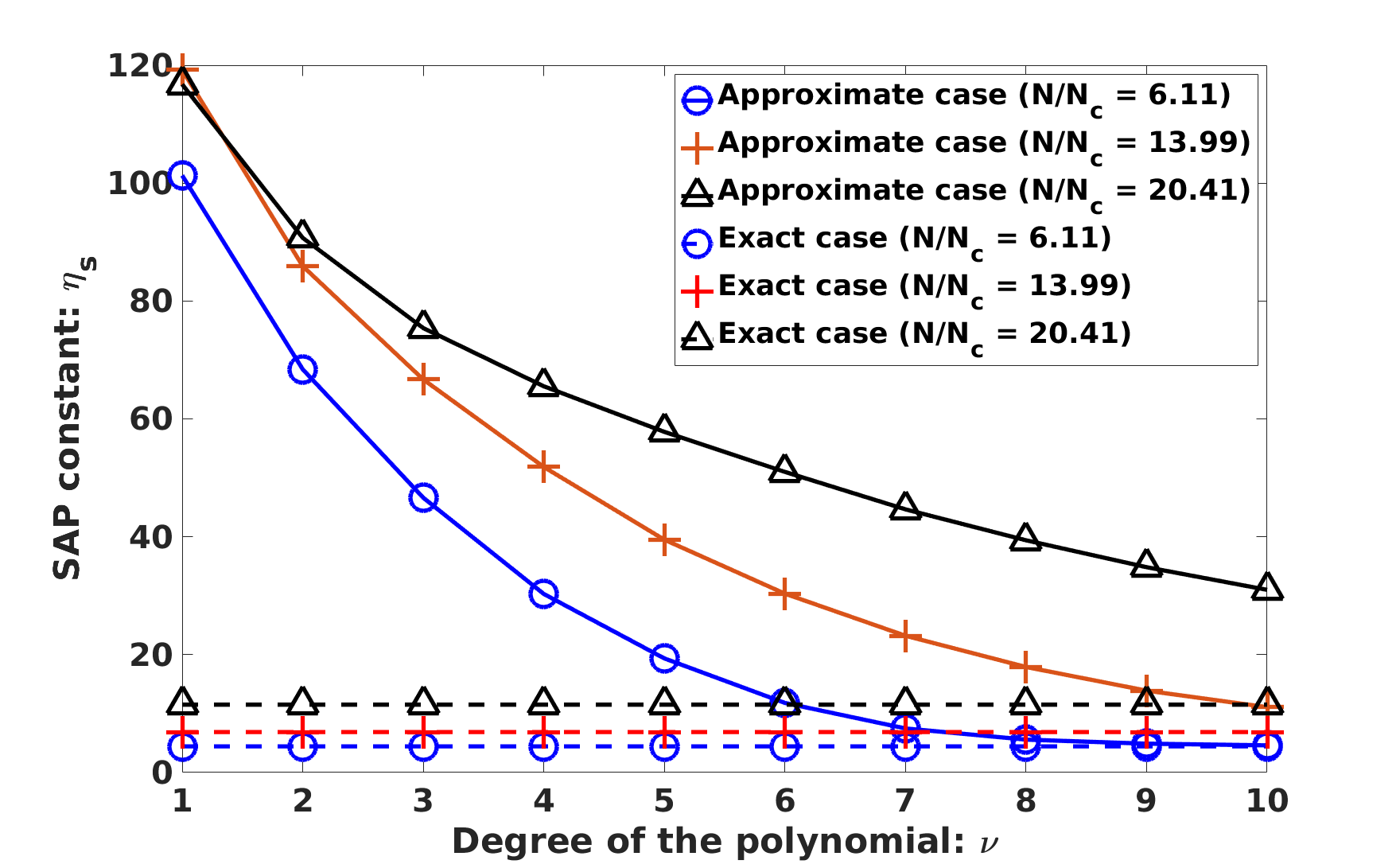}
\caption{Example~\ref{exp:simple-poisson}: the SAP constants for different~$N_c$ ($h=1/64$, $N = 4,225$ and $\epsilon = 10^{-4}$)}
\label{fig:jump-ratio}
\end{center}
\end{figure}

The next test illustrates the properties of the coarse matrices corresponding to the modified coarse spaces  based on polynomial approximation. More specifically, we are interested in the sparsity of the modified prolongation matrix~${\widetilde P}$ (in terms of percentage w.r.t to the matrix size $NN_c$). We also are interested in the AMG operator complexity (OC) defined as the ratio between the total number of nonzeros of $A$  plus the number of the nonzeros of the coarse-level matrix and the number of nonzeros of $A$.
Note that~$\nu = 0$ corresponds to the original prolongation~$P$ (and respective coarse matrix).  From Table~\ref{tab:jump-sparse},  as expected, we see that both the number of nonzeros and operator complexity grow when $\nu$ increases.  The number of nonzeros of ${\widetilde P}$ grows faster when the ratio $\frac{N}{N_c}$ gets bigger whereas the operator complexity actually grows slower when $\frac{N}{N_c}$ gets larger. We note that in practice, for upscaling purposes, we need to have operator complexity less than two (then we use less memory to store the coarse matrix than the original fine-level one).
Our results indicate that to achieve desired approximation accuracy for a reasonable computational cost can be  a challenging task. In addition, we also use the modified coarse space in AMG iterative method and report number of iterations of the two-grid algorithms. Here, we choose $f=1$ in the diffusion problem~\eqref{eqn:jump}. In the two-grid algorithm, Gauss-Seidel relaxation is used, with zero initial guess and the stopping criterion is achieving a reduction of the $\ell_2$ norm of the relative residual by $10^{-6}$. As expected, the number of iterations (Iter) decreases as $\nu$ increases.  We note that in practice for solving linear systems,  we need to consider the trade-off between the computational complexity and convergence behavior. The latter can also be a challenge in practice.
 

\begin{table}[htp]
	\caption{Example~\ref{exp:simple-poisson}: sparsity of the modified coarse space and performance of two-grid AMG method with different $\nu$ ($h=1/64$, $N = 4,225$ and $\epsilon = 10^{-4}$)}
	\begin{center}
		{\footnotesize
			\begin{tabular}{|c||c|c|c|c|c|c|c|c|c|}
				\hline \hline 
				&\multicolumn{3}{|c|}{$N/N_c = 6.11$} & \multicolumn{3}{|c|}{$N/N_c = 13.99$}& \multicolumn{3}{|c|}{$N/N_c = 20.41$} \\ \hline
				&  nnz of ${\widetilde P}$  & OC & Iter &  nnz of ${\widetilde P}$ & OC & Iter &  nnz of ${\widetilde P}$  & OC & Iter \\ \hline
				$\nu = 0$ & $0.15\%$  & $1.22$ &$43$& $0.43\%$  & $1.13$  &$52$& $1.2\%$   & $1.15$ &$60$  \\
				$\nu = 1$ & $0.92\%$  & $2.19$ &$30$& $2.75\%$  & $1.63$  &$42$& $6.97\%$  & $1.65$ &$55$ \\
				$\nu = 2$ & $2.52\%$  & $3.92$ &$23$& $7.27\%$  & $2.41$  &$38$& $17.55\%$ & $2.24$ &$50$ \\
				$\nu = 3$ & $4.92\%$  & $6.62$ &$19$& $13.78\%$ & $3.29$  &$34$& $31.35\%$ & $2.73$ &$48$ \\
				$\nu = 4$ & $8.14\%$  & $8.89$ &$16$& $21.87\%$ & $4.10$  &$30$& $45.87\%$ & $2.98$ &$46$ \\
				$\nu = 5$ & $12.07\%$ & $11.71$&$14$& $30.96\%$ & $4.74$  &$28$& $60.04\%$ & $3.10$ &$41$ \\
				\hline \hline
			\end{tabular}
		}
	\end{center}
	\label{tab:jump-sparse}
\end{table}%

\begin{example}\label{exp:graph-laplacian}
To stress upon the fact that our approach is in fact purely algebraic, we apply our results to graph Laplacian systems corresponding to graphs listed in Table~\ref{tag:graph-laplacian}.
\begin{table}[htp]
	\caption{A set of networks from different real-world applications (first three graphs are from Stanford Large Network Dataset Collection~\cite{SNAP} and the last graph is from SuiteSparse Matrix Collection~\cite{SuiteSparse}). For each graph, we show its number of vertices, number of edges, average vertex degree (ave. deg.) and maximal vertex degree (max. deg.)}\label{tab:snapnets}
	\begin{center}
	{\footnotesize
		\begin{tabular}{|c||c|c|c|c|l|} 
			\hline \hline
			&   Vertices   &  Edges    &  ave. deg. & max. deg. &  \qquad \quad Description \\ \hline \hline
bitcoin-alpha     &  3,775  &   14,120  &  7.48  & 510  & {\tiny Bitcoin Alpha web of trust network}  \\
ego-facebook  & 4,039  &   88,234 &  43.69  & 1045  & {\tiny Social circles from Facebook} \\
ca-GrQc  	  &  4,158  &   13,425  &  6.46   & 81 & {\tiny Collaboration network of Arxiv} \\
rw5151      &  5,151 &  15,248  &  5.92  &  7  &  {\tiny  Markov chain modeling}\\
			\hline \hline
		\end{tabular}
		}
	\end{center}
\label{tag:graph-laplacian}
\end{table}%
\end{example}

\begin{figure}[htbp]
\begin{center}
\includegraphics[width=0.9\textwidth]{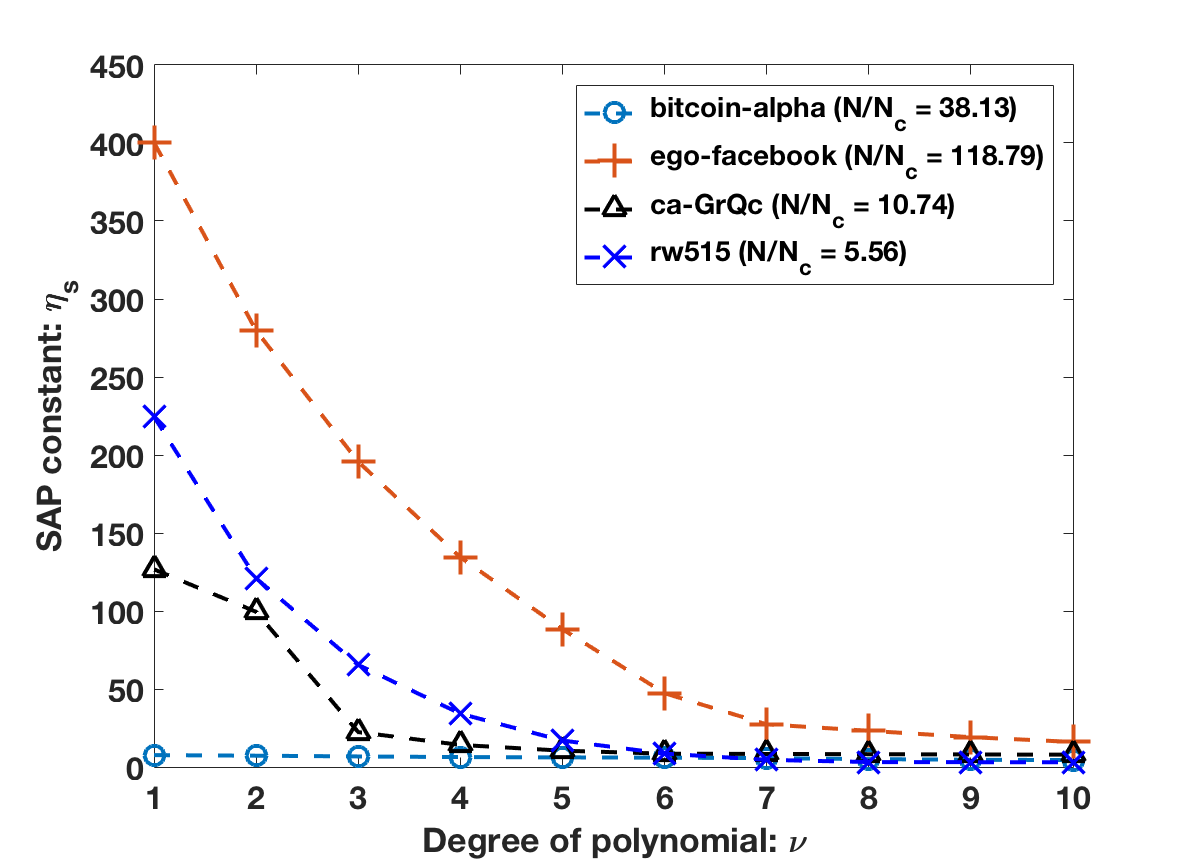}
\caption{Example~\ref{exp:graph-laplacian}: the SAP constants for different~$\nu$}
\label{fig:graph-agg2}
\end{center}
\end{figure}

In Figure~\ref{fig:graph-agg2}, we present the SAP constants for the different graphs from Table~\ref{tag:graph-laplacian}.  Here, we use a simple  unsmoothed aggregation approach. In order to achieve aggressive coarsening, the aggregates are built based on the sparsity pattern of $L^2$, where $L$ corresponds to the graph Laplacian. The original coarse space (or respective interpolation matrix $P$) is constructed using the spectral AMGe method (as used in \cite{hu vassilevski and xu}). As we can see, although the ratio $\frac{N}{N_c}$ differs for the different graphs, if we use relatively accurate approximation (i.e. relatively large $\nu$), the SAP constant stays  small and is fairly similar for different graphs.  This demonstrate that the modified coarse spaces are also robust for these real-world graphs.  

In Figure~\ref{fig:graph-nnzP-agg2} and~\ref{fig:graph-OC-agg2}, we illustrate the sparsity of the modified prolongations and respective coarse matrices.  We notice that the nonzeros percentage of ${\widetilde P}$ grows fairly quickly, which suggests that in practice, only small $\nu$ makes sense. If the coarse level problem are meant to be used multiple times, due to reasonable operator complexity and good approximation property achieved by large~$\nu$, we could use more accurate approximated modified coarse spaces coming from relatively large $\nu$.  For graphs with irregular degree distribution, the challenge to maintain reasonable sparsity of  the coarse matrices with good approximation properties is much more pronounced than in the discretized PDE case and it requires more specialized study.

\begin{figure}[htbp]
\begin{center}
\includegraphics[width=0.9\textwidth]{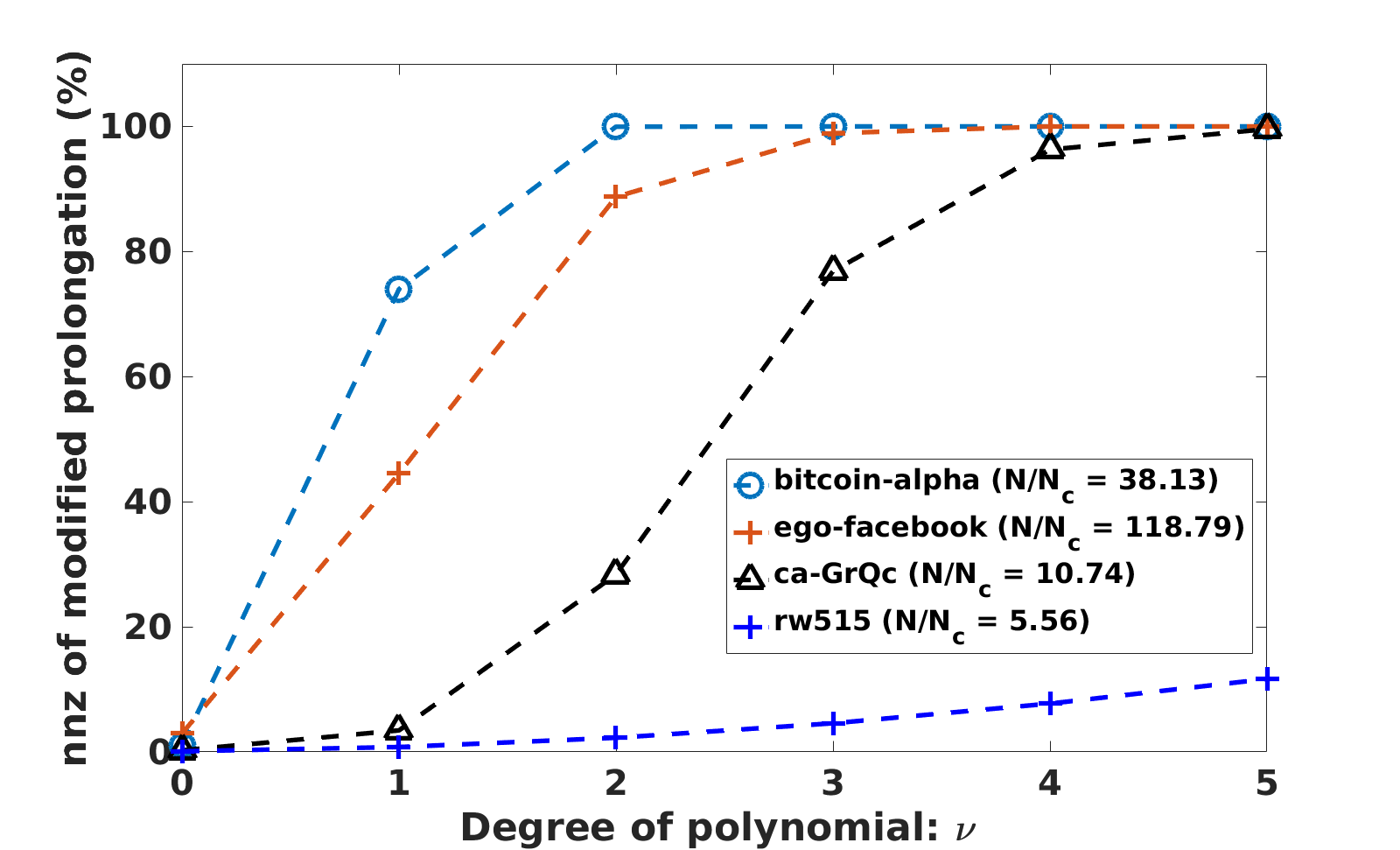}
\caption{Example~\ref{exp:graph-laplacian}: number of nonzeros of ${\widetilde P}$ (in percentage) for different~$\nu$}
\label{fig:graph-nnzP-agg2}
\end{center}
\end{figure}

\begin{figure}[htbp]
\begin{center}
\includegraphics[width=0.9\textwidth]{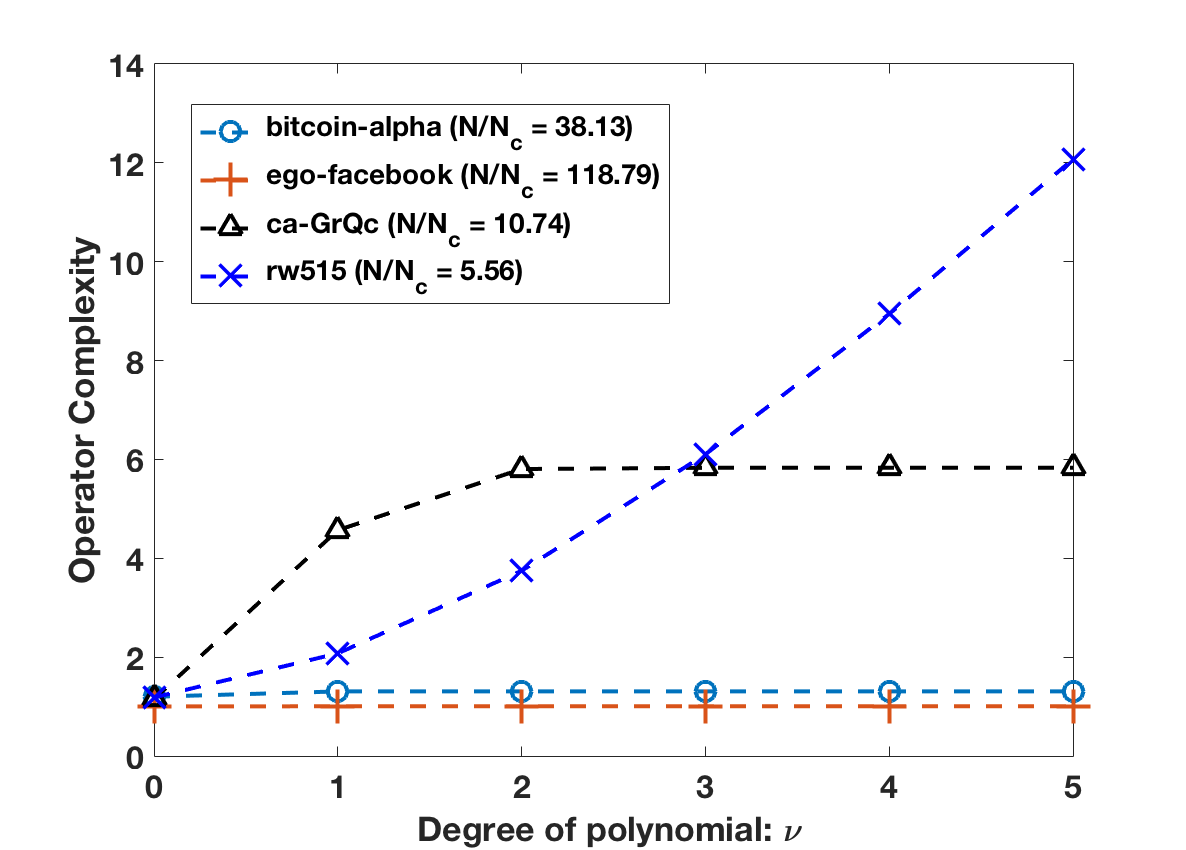}
\caption{Example~\ref{exp:graph-laplacian}: operator complexity for different~$\nu$}
\label{fig:graph-OC-agg2}
\end{center}
\end{figure}

\section{Conclusions}\label{section: conclusions}
In this paper, we investigate the use of certain AMG coarse spaces for the purpose of dimension reduction which in the present setting is referred to as numerical upscaling.  As it is well-understood that although the traditional AMG coarse spaces do  satisfy the WAP (weak approximation property), it is not sufficient for the purpose of upscaling because  the coarse-level solutions do not necessarily approximate the fine-level solution with guaranteed accuracy.  
To remedy this, we follow the approach developed in~\cite{MPe14} extending it to the presented AMG setting. 
The method exploits a projection $\pi_f$ used to modify the original coarse space, which is assumed to possess a WAP, 
so that the resulting new, modified, coarse space satisfies a SAP (strong approximation property) with provable satisfactory bound on the resulting constant $\eta_s$.  More specifically, the  modified coarse space is one of the components in a two-level $A$-orthogonal decomposition so that the corresponding coarse-level solution gives accurate approximation in energy norm.  One main challenge with this approach is the fact that 
the matrix $A_f^{-1}$ used in the definition of $\pi_f$, is dense even if $A_f$ is sparse. Thus, modifying the original  coarse space with exact $\pi_f$ is computationally infeasible (for large-scale problems).  In order to make such modification more practical, we use the fact (which we prove) that $A_f$ is well-conditioned, allowing the use of polynomials to approximate its inverse, leading to an approximate  $\pi_f$, which is used to define an approximate modified coarse space.  Such approximation is computational feasible and also provides provable error estimates in energy norm.  Moreover, the error estimates improve when increasing the degree of the polynomial used in the approximation.  

We provide numerical results that illustrate the theory and demonstrate the accuracy and sparsity of the coarse problems coming from the approximately modified coarse space.  The tests include both, examples of diffusion equation with high contrast coefficients as well as graph Laplacian matrices corresponding to some real-life applications.

As discussed, the use of such modified coarse spaces is of interest in dimension reduction which, as our model tests demonstrate, can be challenging for the present approach (in terms of maintaining reasonable sparsity of the coarse matrices). In the PDE case this challenge seems resolvable if large enough coarsening factor ($H/h$) is employed, whereas in the graph application for graphs with irregular degree distribution, in addition to high coarsening factor one may need to employ graph disaggregation (cf., \cite{disaggregation paper}), which is left for a possible future study.
Additionally, in the PDE case, it is of interest to extend the present results to other types of PDEs such as 
ones posed in  $H(\curl)$ and  $H(\div)$, which will provide alternatives to the existing AMGe upscaling methods (cf., \cite{improved de Rham upscaling}, \cite{Kalchev et al. 2016}, and \cite{BLV}).


\bibliographystyle{siamplain}

\end{document}